\documentclass[11pt,a4paper]{amsart}

\usepackage[T1]{fontenc}
\usepackage[utf8]{inputenc}
\usepackage{lmodern}
\usepackage{amsmath,amssymb,amsthm,mathtools,mathrsfs}
\usepackage{enumitem}
\usepackage{microtype}
\usepackage[margin=1.12in]{geometry}
\usepackage[colorlinks=true,linkcolor=blue,citecolor=blue,urlcolor=blue,hypertexnames=false]{hyperref}

\numberwithin{equation}{section}

\newtheorem{theorem}{Theorem}[section]
\newtheorem{proposition}[theorem]{Proposition}
\newtheorem{lemma}[theorem]{Lemma}
\newtheorem{corollary}[theorem]{Corollary}
\newtheorem{example}[theorem]{Example}
\theoremstyle{definition}
\newtheorem{definition}[theorem]{Definition}
\newtheorem{remark}[theorem]{Remark}

\newcommand{\R}{\mathbb{R}}
\newcommand{\N}{\mathbb{N}}
\newcommand{\E}{\mathbb{E}}
\newcommand{\Pbb}{\mathbb{P}}
\newcommand{\calE}{\mathcal{E}}
\newcommand{\calF}{\mathcal{F}}
\newcommand{\dd}{\,\mathrm{d}}

\newcommand{\ann}{\mathrm{ann}}

\newcommand{\supp}{\operatorname{supp}}

\newcommand{\esssup}{\operatorname*{ess\,sup}}

\newcommand{\eps}{\varepsilon}

\newcommand{\PB}{\mathbb{P}_B}

\begin{document}
\pagestyle{plain}

\title{Small-time annealed large deviation principle for one-dimensional diffusions in a random environment}

\author{Yiduo Wang$^{1}$, Saisai Yang$^{1}$,
		Tusheng Zhang$^{1,2}$}
		
	\footnotetext[1]{\, School of Mathematics, University of Science and Technology of China, Hefei, China. Email: wangyiduo@mail.ustc.edu.cn (Yiduo Wang), yangss@ustc.edu.cn  (Saisai Yang).}
	\footnotetext[2]{\, Department of Mathematics, University of Manchester, Manchester M13 9PL, United Kingdom. Email: tusheng.zhang@manchester.ac.uk }

\begin{abstract}
In this paper, we establish a small-time annealed path large deviation principle for one-dimensional diffusions in a random environment associated with the generator
\begin{equation*}
    {\mathcal L}_W f(x)=e^{-\rho(x,W)}\bigl(e^{a(x,W)}f'(x) \bigr)',
\end{equation*} 
where the coefficients $\{\rho(x,\cdot) : x \in  \R\}$ and $\{a(x,\cdot) : x \in  \R\}$ are random. We assume that for each fixed realization of the environment, $\rho$ and $a$ are continuous and locally exponentially integrable, and that the support of the associated intrinsic coordinates is compact and non-collapsing. This framework includes the extensively studied Brox diffusion:
\begin{equation*}
    \dd X_t=\dd B_t-\frac12\dot W(X_t)\dd t,
\end{equation*}
where $B$ is a standard Brownian motion and $W$ is an independent two-sided Brownian motion representing the environment.

The Itô–McKean representation of the diffusions and the estimates of the first exit probabilities derived via Moser iteration play a crucial role.

\par\vspace{3mm}

\end{abstract}

%
%

	\maketitle
	
	\section{Introduction}
	
	There has been extensive study on Markov processes in random environments since the seminal work of Sinai \cite{Sinai1982} on random walks in random media. In particular, limit theorems for such processes have been studied by many authors; see, for example, \cite{Comets,Greven,He,Spiliopoulos,Varadhan2003} for large deviations and \cite{Fehrman,Guo,Lejay,Rhodes} for homogenization theory.
	
	One important model of diffusions in random environments is the Brox diffusion, which serves as the continuous-time and continuous-space analogue of Sinai's random walk. Formally, it is described by the stochastic differential equation
	\begin{equation}\label{eq:formal-sde}
		\dd X_t=\dd B_t-\frac12\dot W(X_t)\dd t,
	\end{equation}
	where $B$ is a standard Brownian motion and $W$ is an independent two-sided Brownian motion representing the environment. Since $\dot W$ is a spatial white noise, \eqref{eq:formal-sde} is not a classical stochastic differential equation in any sense.
	
	The purpose of this work is to establish the small-time annealed path large deviation principle (LDP) for one-dimensional diffusions in random environments, in particular for  the Brox diffusion. More precisely, we consider the diffusion process $X^W$ associated with the symmetric operator
	\begin{equation*}
		{\mathcal L}_W f(x)=e^{-\rho(x,W)}\bigl(e^{a(x,W)}f'(x) \bigr)',
	\end{equation*} 
	where $W$ represents the random environment belonging to some probability space.
	Our main goal is to prove a small-time annealed LDP for the rescaled process
	\[
	X_{\varepsilon}(t)=X^W(\varepsilon t),\qquad 0\le t\le 1.
	\]
	The quenched LDP concerns the behavior of the diffusion conditioned on a fixed realization of the environment $W$, while the annealed principle averages over the environmental randomness.

	\vskip 0.3cm
	
	One-dimensional diffusions in Brownian potentials have been studied from various perspectives. The localization phenomenon was first established by Brox~\cite{Brox1986}, who introduced the model as a continuous-time analogue of Sinai's random walk~\cite{Sinai1982}. In particular, the process exhibits a characteristic $\log^2 t$ scaling, which starkly contrasts with the diffusive scaling of classical Brownian motion. Long-time asymptotic properties, including transient behavior and stable limit laws, were subsequently examined by Kawazu and Tanaka~\cite{KawazuTanaka1997} and by Hu, Shi, and Yor~\cite{HuShiYor1999}. On the analytical side, the pathwise well-posedness of the singular stochastic differential equation formally associated with the Brox diffusion was established by Hu, Lê, and Mytnik~\cite{HuLeMytnik2017}, while the corresponding quenched and annealed heat kernel estimates were obtained by Chen and Wang \cite{ChenWang}.
		
	Large deviations for one-dimensional random walks in random environments were first studied by Greven and den Hollander~\cite{Greven} in the quenched setting. Since then, large and moderate deviations for diffusions in Brownian potentials have also attracted considerable attention, primarily in the long-time regime. Both quenched and annealed large deviation principles as $t\to\infty$ have been established; see, for example, Taleb~\cite{Taleb2001}, Hu and Shi~\cite{HuShi2004}, and Faraud~\cite{Faraud2011}.
	
	On the other hand, the study of small-time asymptotics of diffusion processes originated with Varadhan~\cite{Varadhan1967,Varadhan67} and was later developed within the Freidlin--Wentzell framework. Subsequently, small-time LDPs have been obtained for a wide class of stochastic systems. In~\cite{Chen}, the authors prove small-time LDPs for Brownian motion with irregular drift by showing that its law is exponentially equivalent to that of Brownian motion. In~\cite{Zhang1}, the third author derives small-time LDPs for infinite-dimensional stochastic evolution equations. And as an application, the small-time LDP can be used to derive Varadhan-type small-time asymptotics for the transition density. In contrast, much less is known for diffusions in random environments. The irregularity of the Brownian potential hinders the  application of classical small-time techniques. To the best of our knowledge, a small-time annealed path LDP for diffusions in random environment has not been studied in the literature. The present work aims to fill this gap.

	\vskip 0.3cm
	
	We emphasize that, while the contraction principle for large deviations and exponentially equivalent transformations for singular drifts have proved useful in previous small-time LDP studies, they are not applicable in our framework, owing to the lack of regularity assumptions on the coefficients and the inherent randomness of the environment. To establish the lower bound of the LDP, we rely on the It\^o--McKean representation of the diffusions and restrict the random environment to a suitable set with positive probability. For the upper bound, we employ Moser iteration to derive a local Davies--Grigor'yan type estimate for the one-dimensional Dirichlet form
	\[
	\calE_W(f,g)=\int_{\R} f'(x)g'(x)e^{a(x,W)}\,\dd x,
	\]
	which yields the estimates of the first exit probabilities from a ball in terms of the intrinsic distance. This estimate is then averaged over the environment to obtain the desired result. Our methodology avoids differentiability assumptions on the coefficients by relying on Dirichlet form theory; in addition to local exponential integrability, the upper-bound argument uses compactness and non-collapse of the environmental intrinsic coordinates.
	
	\vskip 0.3cm
	
	We now introduce some notation used throughout the article.
	\begin{itemize}
		\item Let $C(\mathbb{R})$ be the space of all continuous real-valued functions on $\mathbb{R}$, endowed with the topology of uniform convergence on compact sets, and let $C_0([0,1])$ be the space of continuous functions $f:[0,1]\to\mathbb{R}$ satisfying $f(0)=0$, equipped with the uniform norm $\|f\|_{\infty}:=\sup_{0\le t\le 1}|f(t)|$. 
		\item For an interval \(D\), \(C^1(D)\) denotes the set of continuously differentiable functions on \(D\), and \(C_c^1(D)\) is the subset of compactly supported functions in \(C^1(D)\). Moreover, \(H_0^1(D)\) stands for the Sobolev space over \(D\).
		\item For a strictly increasing function \(f\), we denote its inverse by \(f^{-1}\).
		\item The symbol \(C \) denotes a generic positive constant, independent of \(W\), whose value may change from line to line.
	\end{itemize}
	
	\vskip 0.3cm
	
	The remainder of this paper is organized as follows. Section~\ref{sec:main-results} presents the basic hypotheses and states the main result on the path-space small-time LDP. Furthermore, we prove that the associated rate function of the LDP is a good rate function and provide some examples at the end of the section. Section~\ref{sec:annealed-lower} establishes the annealed lower bound via the Itô–McKean representation and Schilder's lower bound. In Section~\ref{sec:local-exit-proof}, we derive quenched first exit time estimates using Moser iteration. Based on these estimates, we obtain exponential tightness in Subsection~\ref{sec:exponential-tightness} and establish the finite-dimensional upper bounds in Subsection~\ref{Finite-dimensional}. Finally, with these results at hand, we prove the annealed upper bound by a finite open covering argument in Subsection~\ref{Path-space}, thereby completing the proof of the LDP. In the Appendix we provide several auxiliary results used in the main text.

	\section{Preliminaries and main results}\label{sec:main-results}
	Let $(\Omega, \mathcal{F}, \Pbb)$ be a given probability space; an element $W \in \Omega$ is called an environment. We will use the symbol $\E$ to denote the expectation with respect to $\Pbb$. Consider $\R$-valued functions $\rho$ and $a$ defined on $\R \times \Omega$, which are $ \mathcal{B}(\R) \times \mathcal{F}$-measurable. For a fixed realization of $W$, let $X^W$ be a one-dimensional diffusion process associated with the generator
	\begin{equation*}
		{\mathcal L}_W f(x)= e^{-\rho(x,W)}\bigl(e^{a(x,W)}f'(x) \bigr)'.
	\end{equation*} 
	We now list the precise assumptions on the coefficients.
	\vskip 0.3cm
	\noindent \textbf{Hypothesis ($H_{1}$) (Ellipticity and conservativeness).}
	
	\hangindent=2em \hangafter=1
	For $\Pbb$-a.e. $W \in \Omega$, the coefficients $a(\cdot, W)$ and $\rho(\cdot, W)$ are continuous. Moreover,
	\begin{align} \label{2.2}
		\int_0^\infty e^{-a(z,W)}\dd z = \int_{-\infty}^0 e^{-a(z,W)}\dd z = +\infty, \ \ \Pbb\text{-a.e.}
	\end{align}
	
	\noindent \textbf{Hypothesis ($H_{2}$) (Exponential integrability).}
	
	\hangindent=2em \hangafter=1
	For every $c,K \ge 1$,
	$$\E\left [\exp\left(c\sup_{|x|\le K} \left(| \rho(x, \cdot)|+|a(x,\cdot)| \right ) \right)\right]<\infty.$$
	
	\noindent \textbf{Hypothesis ($H_{3}$) (Exponential boundedness).}
	
	For each environment $W$, define
\[
	\Lambda_W(x):=\int_0^x \sqrt{\frac{\exp(\rho(z,W)-a(z,W))}{2}} \,\dd z,
	\qquad x\in\R.
\]
We assume that
	\begin{align} \nonumber
	    \lim_{R\to+\infty} \operatorname*{ess\,inf}_{\substack{W\in\Omega}} \Lambda_W(R)=+\infty, \ \ \ \lim_{R \to - \infty} \esssup\limits_{W \in \Omega} \Lambda_W(R)= - \infty.\nonumber
	\end{align}

\noindent \textbf{Hypothesis ($H_{4}$) (Compactness and non-collapse of the intrinsic coordinates).}

We assume that the set \[
	\mathscr L:=\supp\operatorname{Law}(\Lambda_W)
	\subset C(\R).
\] is compact in $C(\R)$, and that every $\Lambda\in\mathscr L$ is strictly increasing.

\begin{remark}
	\begin{enumerate}
		\item[(1)] In the case of Brox diffusion,
		$\rho(x,W)=\log 2-W(x)$ and $a(x,W)=-W(x)$. Then
		all the conditions $H_1-H_4$ are satisfied.
		\item[(2)] By Feller's test for non-explosion, condition \eqref{2.2}
		implies conservativeness of the process $X^W$.
		\item[(3)] Assumption $H_3$ is mild in the sense that if the diffusion
		process $X^W$ is independent of the environment and $\rho\equiv0$, then
		the exponential boundedness condition
		\[
			\lim_{R\to\infty}\limsup_{\varepsilon\downarrow0}\varepsilon
			\log P_0 \left(\sup_{0\le t\le\varepsilon}|X^W(t)|>R\right)=-\infty
		\]
		implies that
		\[
			\int_0^{\infty}e^{-a(z,W)/2}\,\dd z
			=
			\int_{-\infty}^{0}e^{-a(z,W)/2}\,\dd z
			=+\infty.
		\]
\item[(4)] Assumption $H_4$ is obviously satisfied if the $\rho(x,W)-a(x,W)$ is uniformly bounded from above and below. We would like to point out that the strict increase of each $\Lambda\in\mathscr L$ yields that, for $\Pbb$-a.e. $W\in\Omega$, the topology induced by the intrinsic distance $d_W$ (defined in \eqref{eq:distance} below) coincides with the Euclidean topology.

	\end{enumerate}
\end{remark}
\vskip 0.3cm
	Under assumption $H_{1}$, the diffusion process $X^W$ can be explicitly constructed via the scale transformation and the time change of Brownian motion. In detail, let $B:=(B_t)_{t\ge 0}$ be a one-dimensional standard Brownian motion starting from zero on some given probability space, independent of the random environment $(\Omega, \mathcal{F}, \Pbb)$. For a fixed realization of $W$, define for each $x\in\R$ and $t \ge 0$,
	\begin{align}\nonumber
		&S_W(x):=\int_0^x e^{-a(z,W)}\dd z, \\
		&\sigma_W(x) := \sqrt{2}   \exp\left(-\frac{\rho(S_W^{-1}(x),W)+a(S_W^{-1}(x),W)}{2}\right) \nonumber \\
		&\phi_W(t):= \int_0^t  \sigma_W(B(s)+S_W(x))^{-2} \dd s. \nonumber  
	\end{align}
	Then, by Lemma \ref{l3.1} below, the Itô–McKean construction of the quenched diffusion $X^W$ starting from $x$ is given by the scale-time representation
	\begin{equation}\label{eq:brox-time-change-section2}
		X^W(t)=S_W^{-1}\bigl(B(\phi_W^{-1}(t))+S_W(x)\bigr),\qquad t\ge0.
	\end{equation}
	Equivalently, it is the symmetric diffusion associated with the following Dirichlet form. Set
	\begin{equation*}
		\mu_W(dx)=e^{\rho(x,W)}\,dx.
	\end{equation*}
	On $L^2(\R,\mu_W)$, consider the bilinear form
	\begin{equation*}
		\mathcal E_W(f,g)=\int_{\mathbb R} e^{a(x,W)-\rho(x,W)} f'(x)g'(x)  \mu_W(dx),
	\end{equation*}
	with domain
	\begin{equation*}
		\calF_W=\bigl\{f\in L^2(\mu_W): f\text{ is locally absolutely continuous and } e^{\frac{a(\cdot,W)-\rho(\cdot,W)}{2}}f'\in L^2(\mu_W)\bigr\}.
	\end{equation*}
	For a fixed realization of the random environment $W$, we denote by $P_x^W$ the quenched law of the diffusion started from $x$. Its annealed law is
	\[
		P_x^{\ann}(A):=\int_\Omega P_x^W(A)\,\Pbb(\dd W), \qquad A\in\mathcal B\bigl(C_0([0,1])\bigr).
	\]
	
	\vskip 0.3cm
	
	For the annealed upper bound we need quenched exit-probability estimates with explicit dependence on the environment. The argument combines the Davies--Grigor'yan--Takeda method \cite{Grigoryan1994,Takeda1989} with Sturm's local Moser mean-value estimate \cite{Sturm1995}. In particular, all constants must be controlled in terms of
	\begin{equation}\label{01}
		M(K):=\sup_{x\in K}| \rho(x,W)|+ \sup_{x\in K}|a(x,W)|.
	\end{equation}
	The general results in \cite{Takeda1989, Grigoryan1994, Sturm1995} provide the underlying methodology but not the explicit form required for averaging over the random environment. We therefore provide a self-contained proof of this estimate with explicit constants in Section~\ref{sec:local-exit-proof}.
	
	For the actual environment, define the intrinsic distance induced by the Dirichlet form $\mathcal E_W$ by
	\begin{equation}\label{eq:distance}
		d_W(x,y):=|\Lambda_W(y)-\Lambda_W(x)|
		=\left|\int_x^y \sqrt{\frac{\exp(\rho(z,W)-a(z,W))}{2}}\,\dd z\right|.
	\end{equation}
	For $x\in\R$ and $r>0$, let
	\[
		D_W(x,r):=\{y\in\mathbb R:d_W(x,y)<r\}.
	\]
	Let $ \tau_D$ be the first exit time of $X^W$ from an open set $D \subset \R$, given by
	\[ \tau_D:=\inf\{s\ge0:X_s^W\notin D\}.\]
	Then we have the following result on the first exit probability, which will be proved in Section \ref{sec:local-exit-proof}. These estimates of exit probabilities  will play a fundamental role in the upper bound LDP argument.
	
	\begin{theorem}\label{thm:local-exit}
		Under assumption $H_1$, for each $\theta \in (0,1)$, there exists a deterministic constant $
		C_{\theta}>0$
		such that for every environment $W$, one has for any $0<t<1, \ r>0$ and $ x \in \R$
		\begin{equation}\label{eq:local-exit}
			P_x^W(\tau_{D_W(x,r)}\le t)
			\le C_\theta e^{C_\theta M(D_W(x,r))}
			\exp\left\{-a_\theta\frac{r^2}{t}\right\},
		\end{equation}
		where $M(D_W(x,r))$ is defined in \eqref{01} and 
		\begin{equation*}
			a_\theta=\frac{(1-\theta)^2}{2(1+\theta)}.
		\end{equation*}
		In particular, $a_\theta\uparrow 1/2$ as $\theta\downarrow0$.
	\end{theorem}
	
	As a consequence of Theorem~\ref{thm:local-exit}, we obtain the following product estimate.
	
	\begin{proposition}\label{prop:product}
		Under assumptions $H_1$ and $H_2$, for each $n\in\N$ and positive constants $R, \ r_i$, $h_i$, $1\le i\le n$,
		\begin{equation}\label{eq:product-bound}
			\limsup_{\varepsilon\downarrow0}\varepsilon
			\log \E\left[\prod_{i=1}^n \sup_{|x|\le R}P_x^W\left(\tau_{D_W(x,r_i)}\le \varepsilon h_i ; \sup_{s\le\varepsilon h_i}|X_s^W|\le R\right)\right]
			\le -\sum_{i=1}^n\frac{r_i^2}{2h_i},
		\end{equation}
		where $\E$ denotes expectation with respect to the environment.
	\end{proposition}
	
	\vskip 0.3cm

	We introduce the rate function associated with the rescaled process
	\begin{equation*}
		X_{\varepsilon}(t)=X^W(\varepsilon t),\qquad 0\le t\le1.
	\end{equation*}
For any absolutely continuous \(h\in C_0([0,1])\) with \(h'\in L^2(0,1)\), define the Cameron--Martin energy by
\[
  \mathcal I_{\mathrm{CM}}(h):= \frac12\int_0^1 |h'(t)|^2\,dt.
\]
Otherwise, set \(\mathcal I_{\mathrm{CM}}(h):=+\infty\).

For a fixed environment \(W\), define the quenched action by
\[
  I_W(f):=\mathcal I_{\mathrm{CM}}(\Lambda_W\circ f),
  \qquad f\in C_0([0,1]).
\]
The annealed action is the support relaxation
\begin{equation}\label{eq:rate-function}
	I(f):=\min_{\Lambda\in\mathscr L}
	\mathcal I_{\mathrm{CM}}(\Lambda\circ f),
	\qquad f\in C_0([0,1]).
\end{equation}
The minimum is attained because $\mathscr L$ is compact and the
Cameron--Martin energy is lower semicontinuous for uniform convergence.

We now show that \(I\) is a good rate function. For this purpose, we first introduce the following simple lemma, which will be used frequently in the sequel. As the proof is routine, we shall omit it.
		\begin{lemma}\label{lem:inverse-stability}
		Let $a<b$ and let $\Lambda_n,\Lambda$ be continuous strictly increasing
		functions on $[a,b]$ such that $\Lambda_n\to\Lambda$ uniformly. Suppose
		that $h_n,h\in C([0,1])$, $h_n\to h$ uniformly, and
		$h_n([0,1])\subset\Lambda_n([a,b])$. Then
		$h([0,1])\subset\Lambda([a,b])$ and
		\[
		\Lambda_n^{-1}\circ h_n
		\longrightarrow\Lambda^{-1}\circ h
		\quad\text{uniformly}.
		\]
	\end{lemma}
	
	
	\begin{proposition}[Goodness of the annealed rate function]\label{prop:good-rate}
		Assume $H_3$ and $H_4$. Then the function $I$ in
		\eqref{eq:rate-function} is a good rate function on
		$C_0([0,1])$ equipped with the uniform topology.
	\end{proposition}
	\begin{proof}
		We first prove lower semicontinuity. Let $f_n\to f$ uniformly and assume
		that $\liminf_n I(f_n)<\infty$. Choose minimizers
		$\Lambda_n\in\mathscr L$ in \eqref{eq:rate-function}. Compactness of
		$\mathscr L$ gives, along a subsequence,
		$\Lambda_n\to\Lambda$ locally uniformly. Since the paths $f_n$ have a
		common bounded range,
		\[
		\Lambda_n\circ f_n\longrightarrow\Lambda\circ f
		\quad\text{uniformly}.
		\]
		The lower semicontinuity of the Cameron--Martin energy yields
		\[
		I(f)\le\mathcal I_{\mathrm{CM}}(\Lambda\circ f)
		\le\liminf_{n\to\infty}I(f_n).
		\]
		
		We next prove compactness of the level sets. For $R>0$, put
		\[
		m_R^+:=\min_{\Lambda\in\mathscr L}\Lambda(R)=\operatorname*{ess\,inf}_{\substack{W\in\Omega}}\Lambda_W(R),
		\qquad
		m_R^-:=\min_{\Lambda\in\mathscr L}[-\Lambda(-R)]=\operatorname*{ess\,inf}_{\substack{W\in\Omega}}[-\Lambda_W(-R)].
		\]
		Hypothesis $H_3$ implies
		\[
		m_R^+\longrightarrow +\infty,
		\qquad
		m_R^-\longrightarrow +\infty
		\qquad(R\to +\infty).
		\]
		Fix $L<\infty$ and choose $R$ such that
		$\min\{m_R^+,m_R^-\}>\sqrt{2L}$. If $I(f)\le L$ and $\|f\| \ge R$, a minimizing $\Lambda$ gives a Cameron--Martin path
		$h=\Lambda\circ f$ with
		$\max_t|h(t)|>\sqrt{2L}$. Since $h(0)=0$, the Cauchy--Schwarz
		inequality gives
		\[
		\mathcal I_{\mathrm{CM}}(h)
		=\frac12\int_0^1|h'(t)|^2\,\dd t
		\ge\frac12\max_{0\le t\le1}|h(t)|^2>L,
		\]
		contradicting $\mathcal I_{\mathrm{CM}}(h)\le L$. Thus the level set $\{I\le L\}$ is
		uniformly bounded.
		
		Let $(f_n)\subset\{I\le L\}$ and choose minimizers $\Lambda_n$. Put
		$h_n:=\Lambda_n\circ f_n$. Along subsequences,
		$\Lambda_n\to\Lambda$ locally uniformly and $h_n\to h$ uniformly,
		because the Cameron--Martin level set is compact. By the uniform range
		bound and Lemma~\ref{lem:inverse-stability},
		\[
		f_n=\Lambda_n^{-1}\circ h_n
		\longrightarrow\Lambda^{-1}\circ h
		\quad\text{uniformly}.
		\]
		Lower semicontinuity shows that the limit remains in $\{I\le L\}$.
		Hence every level set is compact.
	\end{proof}

Here is the main result.
%
	\begin{theorem}[Small-time annealed path LDP]\label{thm:path-ldp}
		Assume hypotheses $H_1$--$H_4$ hold. Under the annealed law $P_0^{\ann}$, the family $X_\varepsilon=(X_{\varepsilon}(t))_{0\le t\le1}$ satisfies a large deviation principle on $C_0([0,1])$ with speed $\varepsilon^{-1}$ and the good rate function $I$ given in \eqref{eq:rate-function}. Equivalently, for every open set $G\subset C_0([0,1])$ and every closed set $F\subset C_0([0,1])$,
		\begin{align}
			\liminf_{\varepsilon\downarrow0}\varepsilon\log P_0^{\ann}(X_\varepsilon\in G)
			&\ge -\inf_{\varphi\in G}I(\varphi),\label{eq:ldp-lower}\\
			\limsup_{\varepsilon\downarrow0}\varepsilon\log P_0^{\ann}(X_\varepsilon\in F)
			&\le -\inf_{\varphi\in F}I(\varphi).\label{eq:ldp-upper}
		\end{align}
	\end{theorem}
	
	\begin{remark}
		If the environment $W$ is fixed, the quenched small-time LDP can be obtained by the same methodology. Therefore, we focus exclusively on the annealed case.
	\end{remark}
\vskip 0.3cm	
	
	Finally, we close this section by presenting some examples that are covered by our framework.
	
	\begin{example}[Brox diffusion]
Suppose that assumptions~$H_1$ and~$H_2$ hold, and that $a(x,W)-\rho(x,W)$ is independent of~$x$ with
\[
\esssup\limits_{W\in\Omega}\bigl(a(x,W)-\rho(x,W)\bigr)=\alpha<\infty.
\]
Then assumptions~$H_3$ and~$H_4$ are satisfied, and the associated annealed rate function is given by
\[
I(f)=\frac{e^{-\alpha}}{4}\int_0^1 |  f'(t)|^2\,\mathrm{d}t.
\]
In particular, if we take $a(x,W)=-W(x)$ and $\rho(x,W)=\log 2-W(x)$, where $W$ is a two-sided Brownian motion, then $X^W$ becomes the Brox diffusion. In this case all the above assumptions are fulfilled. Consequently, the Brox diffusion satisfies a small-time large-deviation principle with respect to the annealed law. The corresponding rate function then simplifies to
\[
I(f)=\frac12\int_0^1 |  f'(t)|^2\,\mathrm{d}t.
\]
	\end{example}
	
	\begin{example}[Non-symmetric case]
		For each fixed environment $W$, let the one-dimensional diffusion $X^W$ be associated with the non-symmetric generator
		\begin{equation*}
			{\mathcal A}_W f(x)=e^{-\rho(x,W)} \left(  (e^{a(x,W)}f'(x)  )'+b(x,W)f'(x) \right),
		\end{equation*} 
		where the coefficients $\{\rho(x,\cdot) : x \in  \R\}$, $\{a(x,\cdot) : x \in  \R\}$ and $\{b(x,\cdot) : x \in  \R\}$ are random. Set
		\[
			G(x,W):=\int_0^x b(z,W)e^{-a(z,W)}\,\dd z.
		\]
		Then
		\[
			{\mathcal A}_W f(x)
			=e^{-\rho(x,W)-G(x,W)}
			\bigl(e^{a(x,W)+G(x,W)}f'(x)\bigr)'.
		\]
		If the transformed coefficients $a+G$ and $\rho+G$ satisfy
		$H_1$--$H_4$, then the theorem applies. Since
		$(\rho+G)-(a+G)=\rho-a$, the rate function
		is unchanged.
	\end{example}
	
	\begin{example}[Non-uniform boundedness with respect to the environment]
		Take a deterministic continuous function $\sigma_W: \R \to [0,\infty)$ satisfying $\inf_{x \in [-1,1]} \sigma_W(x) =0$, and $x \mapsto \int_{0}^{x}   \sigma_W(z)     \dd z$ is strictly increasing.
%
%
%
%
    Let $a(x,W):=-2\log(\sigma_W(x)+\xi(W))$ and $\rho\equiv 0$, where $\xi:\Omega\to(0,\infty)$ is a bounded random variable such that $\operatorname*{ess\,inf}\xi=0$ and $\mathbb E[\xi^{-p}]<\infty$ for any $p\ge1$. Then the generator $\mathcal L_W$ satisfies assumptions $H_1$--$H_4$. However,
    \[
        \operatorname*{ess\,inf}_{W\in\Omega}\inf_{x\in[-1,1]} a(x,W) = -\infty.
    \]
    	\end{example}

\section{Annealed lower bound}\label{sec:annealed-lower}
	In this section, we first show the generalized Schilder's lower bound. Then, based on this estimate, we derive the lower bound in \eqref{eq:ldp-lower} via It\^o--McKean representation.

	Let
	\[
	B_\eps(s):=B(\eps s),\qquad s\ge0,
	\]
	where $B$ is the zero-started Brownian motion in the scale--time construction.
	 We shall show the following uniform version of the classical Schilder's lower bound.
	
	\begin{lemma}[Uniform Cameron–Martin tube lower bound]\label{lem:uniform-schilder}
		Fix positive constants $T,a,r$. Let
		\[
		\mathcal H_T(a)=\left\{h\in C_0([0,T]):
		\frac12\int_0^T |h'(s)|^2\dd s\le a\right\}.
		\]
		Then
		\begin{equation}\label{eq:uniform-schilder}
			\liminf_{\eps\downarrow0}\eps\log
			\inf_{h\in\mathcal H_T(a)}
			\PB\left(\sup_{0\le s\le T}|B(\eps s)-h(s)|<r\right)
			\ge -a.
		\end{equation}
	\end{lemma}
	
	\begin{proof}
		Since $\mathcal H_T(a)$ is compact in $C_0([0,T])$, one can choose $h_1,\ldots,h_m\in\mathcal H_T(a)$ such that
		\[
		\mathcal H_T(a)\subset \bigcup_{j=1}^m
		\left\{h:\|h-h_j\|_\infty<\frac r2\right\}.
		\]
		Consequently,
		\[
		\inf_{h\in\mathcal H_T(a)}\PB(\|B_\eps-h\|_\infty<r)
		\ge
		\min_{1\le j\le m}\PB\left(\|B_\eps-h_j\|_\infty<\frac r2\right).
		\]
		By Schilder's theorem on $C_0([0,T])$,
		\[
		\liminf_{\eps\downarrow0}\eps\log
		\PB\left(\|B_\eps-h_j\|_\infty<\frac r2\right)
		\ge -\frac12\int_0^T |  h_j'(s)|^2\dd s
		\ge -a.
		\]
		Taking a finite minimum preserves the same lower bound, which proves \eqref{eq:uniform-schilder}.
	\end{proof}
	
	Next result is  the lower bound of the small-time LDP. Recall that the rate function $I$ is defined in \eqref{eq:rate-function}.
	
	\begin{proposition}\label{prop:annealed-tube-lower}
	Assume hypotheses $H_1$ and $H_4$. Then for every open set
	$G\subset C_0([0,1])$,
	\[
		\liminf_{\varepsilon\downarrow0}\varepsilon\log
		P_0^{\ann}(X_\varepsilon\in G)
		\ge-\inf_{\varphi\in G}I(\varphi).
	\]
\end{proposition}

\begin{proof}
	It is enough to prove that, for every $f\in C_0([0,1])$ with
	$I(f)<\infty$ and every $\eta \in (0,1)$,
	\begin{equation}\label{eq:tube-lower}
		\liminf_{\varepsilon\downarrow0}\varepsilon\log
		P_0^{\ann}\left(\|X_\varepsilon-f\|_\infty<\eta\right)
		\ge-I(f).
	\end{equation}
	Choose a minimizer $\Lambda\in\mathscr L$ in
	\eqref{eq:rate-function} and put
	\[
		h:=\Lambda\circ f.
	\]
	Then $\mathcal I_{\mathrm{CM}}(h)=I(f)$. Extend $f$ and $h$ constantly to
	$[1,\infty)$.

	Choose $R_1>\|f\|_\infty+1$. The compact set $h([0,1])$ lies in the
	interior of $\Lambda([-R_1,R_1])$. Since $\Lambda$ belongs to the
	support of the law of $\Lambda_W$, for every sufficiently small
	$\delta>0$ the event
	\[
		A_\delta:=\left\{W:
		\sup_{|x| \le R_! }|\Lambda_W(x)-\Lambda(x)|<\delta\right\}
	\]
	has positive probability. For $W\in A_\delta$, define the recovery path
	\[
		f_W:=\Lambda_W^{-1}\circ h.
	\]
	After decreasing $\delta$, Lemma~\ref{lem:inverse-stability} gives $f_W([0,1])\subset[-R_1,R_1]$ and
		\begin{equation}\label{f_W}
\sup_{W\in A_\delta}\|f_W-f\|_\infty<\frac\eta2.
	\end{equation}
		

	For $W\in A_\delta$, by the definitions of $S_W$ and $\Lambda_W$, writing $F_W(t):=S_W(f_W(t))$, it follows that
	\begin{equation}
		F_W(t)=\int_0^t \sigma_W(F_W(s))h'(s) \dd s, \ \forall t \ge 0.
	\end{equation}
	Set for each $t \ge 0$,
	\begin{align*}
		& T_W(t):=\int_0^t \sigma_W(F(s))^2 \dd s,\\
		& H_W(t):=F_W(T_W^{-1}(t)).
	\end{align*}
	Since the range of $F_W$ is compact, $T_W$ is strictly increasing, so its inverse is well-defined. Differentiating the identity
	\begin{align*}
		t=\int_0^{T_W^{-1}(t)} \sigma_W(F_W(s))^2 \dd s,
	\end{align*}
	we get
	\begin{equation} \nonumber
		T_W^{-1}(t)=\int_0^t \sigma_W(H_W(s))^{-2} \dd s, \ \forall t \ge 0.
	\end{equation}
	Therefore,
	\begin{equation} \label{H}
		H_W(t)=F_W\left(\int_0^t \sigma_W(H_W(s))^{-2} \dd s\right), \ \forall t \ge 0.
	\end{equation}
	Moreover,
	\begin{align*}
		H_W'(t)
		&=  F_W'\left(\int_0^t \sigma_W(H(s))^{-2} \dd s\right) \sigma_W(H_W(t))^{-2}\\
		&=  \sigma_W \left(F_W\left(\int_0^t \sigma_W(H(s))^{-2} \dd s\right) \right) h' \left(\int_0^t \sigma_W(H(s))^{-2} \dd s\right) \sigma_W(H_W(t))^{-2}\\
		&=h' \left(\int_0^t \sigma_W(H(s))^{-2} \dd s\right) \sigma_W(H_W(t))^{-1},
	\end{align*}
	hence
		\begin{equation}\label{HW}
			\int_0^\infty|H_W'(u)|^2\,\dd u
		=\int_0^1|h'(t)|^2\,\dd t=2I(f).
		\end{equation}

	For every $W\in A_\delta$, since the coefficients are continuous, one can choose deterministic integers
	$K,M\ge1$ such that the event
		\begin{align*}
		A_{\delta,K,M,\kappa}:= \Bigl\{ W\in A_\delta :\ &    \sup_{0 \le t \le 1}|F_W(t)| \le K,  \sup_{S_W^{-1}(-K-1) \le x \le S_W^{-1}(K+1)}e^{|a(x,W)|} \le M , \\
		& \sup_{|x| \le K+1}   \sigma_W(x)  \le M, \inf_{|x| \le K+1} \sigma_W(x) \ge M^{-1}   \Bigr\}
	\end{align*}
	has positive probability.
	For $r\in(0,1]$, let
	\[
		\omega_{\sigma_W,K}(r):=
		\sup\{|\sigma_W(u)-\sigma_W(v)|:|u|,|v|\le K+1,
		|u-v|\le r\}.
	\]
	Since this modulus tends to zero for every fixed environment, one can
	choose a deterministic $r\in(0,\eta/(4M))$ such that the event
	\[
		A:=A_{\delta,K,M}\cap
		\left\{M^2\|h'\|_{L^2}
		\bigl(2 M^7\omega_{\sigma_W,K}(r)\bigr)^{1/2}
		\le\frac\eta4\right\}
	\]
	has positive probability.

	For $W\in A$, define
	\[
		E_W(r):=\left\{
		\sup_{0\le u\le M^2}|B_\varepsilon(u)-H_W(u)|<r
		\right\}.
	\]
	On $E_W(r)$ one has $|B_\varepsilon(u)|\le K+1$ for $u\le M^2$.
	Since $\sigma_W\le M$ on $[-K-1.K+1]$,
	\[
		\phi_W(\varepsilon M^2)
		=\varepsilon\int_0^{M^2}\sigma_W(B_\varepsilon(u))^{-2}\,\dd u
		\ge\varepsilon,
	\]
	and hence $\phi_W^{-1}(\varepsilon)/\varepsilon\le M^2$. 
	Using \eqref{H}, the
	$M$-Lipschitz bounds for $S_W^{-1}$ on the selected compact
	intervals and
	\[
		|F_W(t)-F_W(s)|
		\le M\|h'\|_{L^2}|t-s|^{1/2},
	\]
	we obtain on $E_W(r)$,
		\begin{align}
		& \ \ \ \ \|X_\varepsilon-f_W\|_\infty \nonumber \\
		&\le M \sup_{0 \le t \le 1} \left| B(\phi_W^{-1}(\varepsilon t)) - F_W(t) \right| \nonumber \\
		&=M \sup_{0 \le t \le \frac{\phi_W^{-1}(\varepsilon)}{\varepsilon}}
		\left| B_{\varepsilon}(t) - F_W\left( \int_0^{t} \sigma_W(B_{\varepsilon}(s))^{-2} \, \dd s \right) \right| \nonumber \\
		&\le M \sup_{0 \le t \le M^2}
		\left| B_{\varepsilon}(t) - H_W(t) \right|
		+M \sup_{0 \le t \le M^2}
		\left| F_W\left(\int_0^t \sigma_W(H_W(s))^{-2} \dd s\right) - F_W\left( \int_0^{t} \sigma_W(B_{\varepsilon}(s))^{-2}  \dd s \right) \right| \nonumber \\
		&\le Mr
		+ M^2\|h'\|_{L^2}
		\bigl(2 M^7 \omega_{\sigma_W,K}(r)\bigr)^{1/2}
		<\frac\eta2 \nonumber .
	\end{align}
	Together with \eqref{f_W}, we obtain that
	$E_W(r)\subset\{\|X_\varepsilon-f\|_\infty<\eta\}$.

	Consequently,
	\[
		P_0^{\ann}(\|X_\varepsilon-f\|_\infty<\eta)
		\ge\Pbb(A)\inf_{W\in A} P^W(E_W(r)).
	\]
	Combining this with \eqref{HW}, Lemma~\ref{lem:uniform-schilder} therefore yields
	\[
		\liminf_{\varepsilon\downarrow0}\varepsilon\log
		\inf_{W\in A}P^W(E_W(r))\ge-I(f).
	\]
	This proves \eqref{eq:tube-lower}.
\end{proof}

	\section{Estimates of the exit probabilities}\label{sec:local-exit-proof}
	In this section, we use Moser iteration to derive quenched exit-probability estimates and to prove Theorem~\ref{thm:local-exit} and Proposition~\ref{prop:product}.
	Throughout this section, the environment $W\in \Omega$ is fixed, and we assume the assumption $H_{1}$ holds. For $R>0$, set
	\begin{equation*}
		M_{R}:=\sup_{x \in [-R,R]} | \rho(x,W)|+\sup_{x \in {[-R,R]}} |a(x,W)| .
	\end{equation*}
	All constants below are deterministic and independent of $R$, and may change from line to line.
	
	Recall that the diffusion $X^W$ is associated with the Dirichlet form on $L^2(\R,\mu_W)$
	\begin{equation*}
		\mathcal E_W(f,g)=\int_{\mathbb R} e^{a(x,W)-\rho(x,W)} f'(x)g'(x)  \mu_W(dx),
	\end{equation*}
	and the intrinsic distance is given by
	\begin{equation} \nonumber
		d_W(x,y)
		=\left|\int_x^y \sqrt{\frac{\exp(\rho(z,W)-a(z,W))}{2}}\,\dd z\right|.
	\end{equation}
	If $x,y\in[-R,R]$, then
	\begin{equation}\label{R}
		\frac{e^{-M_R/2}}{\sqrt2}|x-y|
		\le d_W(x,y)
		\le\frac{e^{M_R/2}}{\sqrt2}|x-y|.
	\end{equation}
	By \eqref{R} and the definition $D_W(x,r)=\{y\in\mathbb R:d_W(x,y)<r\}$, whenever $\overline{D_W(x,r)}\subset {[-R,R]}$,
	\begin{equation}\label{eq:volume}
		c\,e^{-3M_R/2}r
		\le \mu_W(D_W(x,r))
		\le C\,e^{3M_R/2}r,
	\end{equation}
	which yields
	\begin{equation}\label{vr}
		\frac{\mu_W(D_W(x,r))}{\mu_W(D_W(x,s))}
		\le C e^{3M_R}\frac{r}{s},\qquad 0<s\le r.
	\end{equation}
	
	\begin{definition}[Weak sub-solution]\label{def:subcaloric}
		A nonnegative function $u$ on $(0,S)\times D$ is called a bounded weak sub-solution  of the differential equation
		\begin{equation}\label{02}
			\frac{\partial u}{\partial t}={\mathcal L}_W u
		\end{equation}
		if
		\[
		u\in L^\infty((0,S)\times D) \cap L^2_{\mathrm{loc}}((0,S);H^1_{\mathrm{loc}}(D)),
		\]
		and, for every nonnegative $\varphi\in C_c^1((0,S);H_0^1(D)\cap L^\infty(D))$, one has
		\begin{equation*}
			-\int_0^S\!\int_Du\,\partial_s\varphi\,\dd \mu_W\dd s
			+\int_0^S\mathcal E_W(u(s),\varphi(s))\,\dd s\le0.
		\end{equation*}
		If equality holds for all such $\varphi$, then $u$ is called a weak solution.
	\end{definition}

	The next two lemmas provide the Caccioppoli inequality and the one-dimensional weighted Sobolev inequality, which are crucial for the Moser iteration.
	
	\begin{lemma}[Caccioppoli inequality]\label{lem:caccioppoli}
		Let $u\ge0$ be a bounded weak sub-solution to equation~\eqref{02} on $(s_1,s_2)\times I$, where $I$ is a bounded interval. Let $p\ge2$, $\zeta\in C_c^1(I)$, and $\chi\in C_c^1((s_1,s_2))$ satisfy $0\le\zeta,\chi\le1$. Then
		\begin{align}\label{eq:caccioppoli}
			&\operatorname*{ess\,sup}_{s\in[s_1,s_2]}
			\chi(s)^2\int_I\zeta^2u(s)^p\,\dd\mu_W
			+\int_{s_1}^{s_2}\!\chi^2\int_I\zeta^2
			\left|\partial_y(u^{p/2})\right|^2e^{a}\,\dd y\,\dd s\\
			&\quad\le Cp^2\left(\|\chi'\|_{L^\infty(s_1,s_2)}+
			\left\|2e^{a-\rho}(\zeta')^2\right\|_{L^\infty(I)}\right)
			\int_{s_1}^{s_2}\!\int_{\operatorname{supp}\zeta}u^p\,\dd\mu_W\,\dd s.\nonumber
		\end{align}
	\end{lemma}

	\begin{proof}
		Choose $J\Subset(s_1,s_2)$ containing $\operatorname{supp}\chi$, and mollify $u$ in time on $J$:
		\[
			u_h(s,y):=\frac1h\int_{\R}\eta\left(\frac{s-r}{h}\right)u(r,y)\,\dd r,
		\]
		where $\eta\in C_c^\infty((-1,1))$ is a nonnegative mollifier and $h<\operatorname{dist}(J,\{s_1,s_2\})$. Then $u_h$ remains weak subcaloric on $J\times I$ and, for every nonnegative spatial test function $\psi$,
		\begin{equation*}
			\langle\partial_su_h,\psi\rangle_{L^2(\mu_W)}+\mathcal E_W(u_h,\psi)\le0.
		\end{equation*}
		Insert $\psi=\chi^2\zeta^2u_h^{p-1}$. The Sobolev chain rule and Young's inequality give
		\begin{align*}
			&\frac{d}{\dd s}\left(\chi^2\int_I\zeta^2u_h^p\,\dd\mu_W\right)
			+c\chi^2\int_I\zeta^2|\partial_y(u_h^{p/2})|^2e^a\,\dd y\\
			&\qquad\le Cp^2\left(\|\chi'\|_\infty+
			\|2e^{a-\rho}(\zeta')^2\|_{L^\infty(I)}\right)
			\int_{\operatorname{supp}\zeta}u_h^p\,\dd\mu_W.
		\end{align*}
		Here the mixed term is estimated using
		\[
			2|\zeta\zeta'|u_h^{p-1}|\partial_yu_h|
			\le\frac{p-1}{2}\zeta^2u_h^{p-2}|\partial_yu_h|^2
			+\frac{2}{p-1}(\zeta')^2u_h^p.
		\]
		Integrating in time and taking the essential supremum gives the analogue of \eqref{eq:caccioppoli} for $u_h$.

		Let $h_n\downarrow0$. After passing to a subsequence, $u_{h_n}\to u$ a.e. and in $L^2_{\mathrm{loc}}((s_1,s_2);H^1(I))$; boundedness of $u$ and the chain rule imply $u_{h_n}^{p/2}\to u^{p/2}$ in the corresponding local energy space. Fatou's lemma for the essential-supremum term and lower semicontinuity of the energy term then yield \eqref{eq:caccioppoli}.
	\end{proof}
	
	\begin{lemma}\label{lem:sobolev}
		Let $I$ be a bounded interval and let $f\in H_0^1(I)$. Then
		\begin{equation}\label{eq:sobolev1}
			\|f\|_{L^\infty(I)}^2
			\le 2 e^{M(I)}\|f\|_{L^2(I,\mu_W)}\,\mathcal E_W(f,f)^{1/2}.
		\end{equation}
		Consequently,
		\begin{equation}\label{eq:sobolev2}
			\int_I|f|^6\,\dd \mu_W
			\le 16 e^{2M(I)}\left(\int_If^2\,\dd \mu_W\right)^2\mathcal E_W(f,f).
		\end{equation}
	\end{lemma}
	
	\begin{proof}
		It is enough to consider $f\in C_c^1(I)$. By the fundamental theorem of calculus, for every $x \in I$,
		\begin{equation}\label{4.14}
			|f(x)|^2\le2\int_I|f(y)f'(y)|\,\dd y \le2 \left( \int_I f(y)^2 \dd y\right)^{1/2}  \left(\int_I|f'(y)|^2\,\dd y\right)^{1/2} .
		\end{equation}
		Since $\mu_W(\dd y)=e^{\rho(y,W)}\,\dd y$ and
		$\mathcal E_W(f,f)=\int_I e^{a(y,W)}|f'(y)|^2\,\dd y$, the two
		unweighted factors in \eqref{4.14} satisfy
		\[
		\left(\int_I f^2\,\dd y\right)^{1/2}
		\le \sup_{y\in I}e^{-\rho(y,W)/2}
		\left(\int_I f^2\,\dd\mu_W\right)^{1/2}
		\]
		and
		\[
		\left(\int_I |f'|^2\,\dd y\right)^{1/2}
		\le \sup_{y\in I}e^{-a(y,W)/2}\,
		\mathcal E_W(f,f)^{1/2}.
		\]
		Consequently,
		\begin{align*}
		\|f\|_{L^\infty(I)}^2
		&\le 2\sup_{y\in I}\exp\left\{-\frac{\rho(y,W)+a(y,W)}2\right\}
		\|f\|_{L^2(I,\mu_W)}\,\mathcal E_W(f,f)^{1/2}\\
		&\le 2e^{M(I)}\|f\|_{L^2(I,\mu_W)}\,\mathcal E_W(f,f)^{1/2}.
		\end{align*}
		This proves \eqref{eq:sobolev1}. Equation \eqref{eq:sobolev2}
		follows from \eqref{eq:sobolev1} and
		$\int_I|f|^6\,\dd\mu_W\le\|f\|_\infty^4\int_I f^2\,\dd\mu_W$.
	\end{proof}
	
	For $\ell>0$, $t_0>0$, and $x_0\in\mathbb R$, define the intrinsic cylinders
	\[
	I_\lambda=D_W(x_0,\lambda\ell),\qquad
	J_\lambda=(t_0-\lambda^2\ell^2,t_0+\lambda^2\ell^2),\qquad
	Q_\lambda=J_\lambda\times I_\lambda.
	\]
	
	\begin{lemma}[Moser iteration step]\label{lem:moser-step}
		Assume that $Q_1\subset(0,S)\times {[-R,R]}$ and let $u\ge0$ be a bounded weak sub-solution to equation (\ref{02}) on $Q_1$. There is an absolute constant $C$ such that, for every $p\ge2$ and $1/2\le\lambda'<\lambda\le1$,
		\begin{equation}\label{eq:moser-step}
			\|u\|_{L^{3p}(Q_{\lambda'},\mu_W\dd s)}
			\le\left(\frac{Ce^{cM_R}p^2}{(\lambda-\lambda')^2\ell^2}\right)^{1/p}
			\|u\|_{L^p(Q_\lambda,\mu_W\dd s)}.
		\end{equation}
	\end{lemma}
	
	\begin{proof}
		By \eqref{R}, we can choose a spatial cutoff $\zeta$ which equals one on $I_{\lambda'}$, vanishes outside $I_\lambda$, and satisfies
		\begin{equation}\label{eq:cutoff}
			|\zeta'|
			\le\frac{C e^{M_R}}{(\lambda-\lambda')\ell}.
		\end{equation}
		Choose a time cutoff $\chi$ which equals one on $J_{\lambda'}$, vanishes near the ends of $J_\lambda$, and satisfies $|\chi'|\le C((\lambda-\lambda')^2\ell^2)^{-1}$. Put $w:=\chi\zeta u^{p/2}$. Since
		\[
		|\partial_yw|^2
		\le2\chi^2\zeta^2|\partial_yu^{p/2}|^2
		+2\chi^2(\zeta')^2u^p,
		\]
		Lemma~\ref{lem:caccioppoli}, together with \eqref{eq:cutoff}, gives
		\begin{align}\label{eq:moser-energy}
			&\sup_s\int_{I_\lambda}w(s)^2\,\dd \mu_W
			+\int_{J_\lambda}\mathcal E_W(w(s),w(s))\,\dd s\\
			&\qquad\le\frac{Cp^2 e^{4 M_R}}{(\lambda-\lambda')^2\ell^2}
			\int_{Q_\lambda}u^p\,\dd \mu_W\dd s.\nonumber
		\end{align}
		Applying Lemma~\ref{lem:sobolev} to $w(s,\cdot)$ and integrating in time, we obtain
		\begin{equation}\label{eq:space-time-sobolev}
			\int_{Q_\lambda}w^6\,\dd \mu_W\dd s
			\le 16e^{2M_R}\left(\sup_s\int_{I_\lambda}w(s)^2\,\dd \mu_W\right)^2
			\int_{J_\lambda}\mathcal E_W(w(s),w(s))\,\dd s.
		\end{equation}
		Since $w=u^{p/2}$ on $Q_{\lambda'}$, substitution of \eqref{eq:moser-energy} into \eqref{eq:space-time-sobolev} yields
		\[
		\int_{Q_{\lambda'}}u^{3p}\,\dd \mu_W\dd s
		\le\left(\frac{Ce^{cM_R}p^2}{(\lambda-\lambda')^2\ell^2}
		\int_{Q_\lambda}u^p\,\dd \mu_W\dd s\right)^3.
		\]
		Taking the power $1/(3p)$ proves \eqref{eq:moser-step}.
	\end{proof}
	
	\begin{lemma}[Local mean-value estimate]\label{lem:mean1}
		Assume that $Q_1\subset(0,S)\times {[-R,R]}$ and let $u\ge0$ be a bounded weak sub-solution to equation (\ref{02}) on $Q_1$. Then
		\begin{equation}\label{eq:mean1}
			\|u\|_{L^{\infty}(Q_{1/2})}^2
			\le\frac{Ce^{cM_R}}{\ell^2\mu_W(I_1)}
			\int_{Q_1}u(s,y)^2\,\dd \mu_W(y)\dd s.
		\end{equation}
	\end{lemma}
	
	\begin{proof}
		Let $p_n=2\cdot3^n$ and $\lambda_n=1/2+2^{-n-1}$. Applying Lemma~\ref{lem:moser-step} with $p=p_n$, $\lambda=\lambda_n$, and $\lambda'=\lambda_{n+1}$ gives
		\begin{equation}\label{03}
			\|u\|_{L^{p_{n+1}}(Q_{\lambda_{n+1}})}
			\le\left(\frac{Ce^{cM_R}p_n^2 4^n}{\ell^2}\right)^{1/p_n}
			\|u\|_{L^{p_n}(Q_{\lambda_n})}.
		\end{equation}
		Because the three series
		\[
		\sum_{n\ge0}\frac1{p_n},\qquad
		\sum_{n\ge0}\frac n{p_n},\qquad
		\sum_{n\ge0}\frac{\log p_n}{p_n}
		\]
		are finite, iterating (\ref{03})  over $n$ yields
		\begin{equation}\label{eq:Linfty}
			\|u\|_{L^\infty(Q_{1/2})}
			\le Ce^{cM_R}\ell^{-3/2}\|u\|_{L^2(Q_1,\mu_W\dd s)}.
		\end{equation}
		By \eqref{eq:volume}, $\mu_W(I_1)\le2e^{3M_R/2}\ell$. Squaring \eqref{eq:Linfty} and absorbing the additional exponential factor proves \eqref{eq:mean1}.
	\end{proof}
	
	We next prove the weighted $L^2$ estimate for the exit probability. This is the integral maximum principle with respect to the intrinsic distance \eqref{eq:distance}.
	
	\begin{lemma}[Weighted $L^2$ exit estimate]\label{lem:L2exit}
		Take intervals $A \subset D\subset {[-R,R]}$. Assume $v$ is a weak solution to equation (\ref{02}) on $(0,S)\times D$ satisfying $0 \le v \le 1$ and $v(0,x)=0,x \in D$. Assume
		\[
		r:=\inf\{d_W(y,z):y\in A,\ z\notin D\}>0.
		\]
		Then, for every $t \in (0,S)$,
		\begin{equation}\label{eq:L2exit}
			\int_Av(t,y)^2\,\dd \mu_W(y)
			\le C\mu_W(D\setminus A)
			\left(1+\frac{t}{r^2}+\frac{r^2}{t}\right)
			\exp\left\{-\frac{r^2}{t}\right\},
		\end{equation}
		where $C$ is an absolute constant.
	\end{lemma}
	
	\begin{proof}
		Similar to the argument in Lemma \ref{lem:caccioppoli}, without loss of generality, we may assume $v$ is smooth in time. Let $\eta\in C_c^1(D)$ and let $\xi=\xi(s,y)$ satisfy
		\begin{equation}\label{eq:xi-condition}
			\partial_s \xi+\frac{1}{2}e^{a-\rho}(\partial_y \xi)^2\le0.
		\end{equation}
		Set
		\[
		F(s)=\int_D\eta^2v(s,y)^2e^{\xi(s,y)}\,\dd \mu_W(y).
		\]
		Since $v$ solves the heat equation in the weak form
		\begin{equation}\label{413}
			\int_D \partial_s v\psi \dd \mu_W+\int_D e^{a-\rho}\partial_y  v \partial_y  \psi \dd \mu_W=0,
		\end{equation}
		applying \eqref{413} with $\psi:=\eta^2ve^\xi$, we get
		\begin{align*} 
			F'(s)
			&=-2\int_De^{a}\partial_y v \partial_y (\eta^2ve^\xi)\,\dd y
			+\int_De^{a-\rho}\eta^2v^2e^\xi\partial_s \xi\,\dd \mu_W\\
			&=-2\int_De^{a}\eta^2e^\xi (\partial_y v)^2\,\dd y
			-4\int_De^{a}\eta\eta'e^\xi v \partial_y  v\,\dd y\\
			&\quad -2\int_De^{a}\eta^2e^\xi \partial_y \xi v \partial_y v\,\dd y
			+\int_D\eta^2v^2e^\xi\partial_s \xi\,\dd \mu_W \\
			&\le\int_De^\xi v^2\left[
			2e^{a-\rho}(\eta')^2+
			2e^{a-\rho}|\eta\eta' \partial_y \xi|
			+\eta^2\left(\partial_s \xi+\frac{1}{2}e^{a-\rho}(\partial_y \xi)^2\right)
			\right]\dd \mu_W \\
			&\le\int_{\operatorname{supp}\eta'}e^\xi v^2\left[
			2e^{a-\rho}(\eta')^2+
			2e^{a-\rho}|\eta\eta' \partial_y \xi|
			\right]\dd \mu_W,\nonumber
		\end{align*}
		where the last inequality follows from \eqref{eq:xi-condition}. Integrating from $0$ to $t$, using $v(0,\cdot)=0$ in the interior and $0\le v \le 1$, we obtain
		\begin{equation}\label{eq:F-bound-smooth}
			F(t)\le
			\int_0^t\int_{\operatorname{supp}\eta'}e^\xi \left[
			2e^{a-\rho}(\eta')^2+
			2e^{a-\rho}|\eta\eta'\xi_y|
			\right]\mu_W(\dd y)\dd s.
		\end{equation}
		
		Now choose the particular cutoffs. Fix a constant $\alpha\in(0,1/4]$ and a Lipschitz function $\chi_\alpha:[0,\infty)\to[0,1]$
		with
		\[
		\chi_\alpha(x)=0\ (x\le\alpha/2),\qquad
		\chi_\alpha(x)=1\ (x\ge\alpha),\qquad
		|\chi_\alpha'|\le3/\alpha,
		\]
		and put
		\[
		d_D(y):=\inf_{z\notin D}d_W(y,z), \qquad
		\eta(y):=\chi_\alpha\left(\frac{d_D(y)}r\right),
		\qquad
		\xi(s,y):=\frac{d_D(y)^2}{s+\alpha t}.
		\]
		Since
		\begin{equation*}
			2e^{a(y)-\rho(y)}|d_D'(y)|^2\le1
			\quad\text{for a.e. }y\in D,
		\end{equation*}
		condition \eqref{eq:xi-condition} holds. On $\operatorname{supp}\eta'$, one has $d_D\le\alpha r$ and hence
		\[
		e^\xi\le e^{\alpha r^2/t},\qquad
		2e^{a-\rho}(\eta')^2\le\frac{C}{\alpha^2r^2},\qquad
		2e^{a-\rho}|\eta\eta'\xi_y|\le\frac{C}{s+\alpha t}.
		\]
		Substituting this into \eqref{eq:F-bound-smooth}, we get
		\begin{equation}\label{eq:Fupper}
			F(t)\le C\mu_W(D\setminus A)e^{\alpha r^2/t}
			\left(\frac{t}{\alpha^2r^2}+\log(1+\alpha^{-1})\right).
		\end{equation}
		On the other hand, $\eta=1$ and $d_D\ge r$ on $A$, so
		\begin{equation}\label{eq:Flower}
			F(t)\ge\exp\left\{\frac{r^2}{(1+\alpha)t}\right\}
			\int_Av(t)^2\,\dd \mu_W.
		\end{equation}
		Combining \eqref{eq:Fupper} and \eqref{eq:Flower} yields
		\begin{equation}\label{eq:L2-exit-alpha-detailed}
			\int_A v(t)^2\dd \mu_W
			\le C\mu_W(D\setminus A)
			\left(\frac{t}{\alpha^2 r^2}+\log(1+\alpha^{-1})\right)
			\exp\left\{ -\frac{r^2}{t+\alpha t}+\alpha \frac{r^2}{t}\right\}.
		\end{equation}
		If $0<\frac{r^2}{t}\le4$, take $\alpha=1/4$. If $\frac{r^2}{t}>4$, take $\alpha=\frac{t}{r^2}$. Because $-\frac{z}{1+z^{-1}} \le -z+1$, in both cases the right-hand side of \eqref{eq:L2-exit-alpha-detailed} is bounded by
		\[
		C\mu_W(D\setminus A)(1+\frac{r^2}{t}+\frac{t}{r^2})e^{-\frac{r^2}{t}}.
		\]
		This proves \eqref{eq:L2exit}.
	\end{proof}

	We now combine the local mean estimate and the weighted $L^2$ exit estimate to prove Theorem~\ref{thm:local-exit}.
	
	\begin{proof}[{\bf Proof of Theorem~\ref{thm:local-exit}}]
		Fix $\theta\in(0,1)$ and choose $\beta\in(0,\theta)$. Put
		\[
			D:=D_W(x,r),\qquad A_\beta:=D_W(x,\beta r),
			\qquad R_\beta:=(1-\beta)r,
		\]
		and let $v(s,y):=P_y^W(\tau_D\le s)$. Then $v$ is a continuous weak solution on the cylinder $(0,S)\times D$ by Lemma \ref{lem:exit-caloric} below. Set $z:=r^2/t$.
		
		If $M(D)=\infty$, the desired estimate is trivial. Hence assume $M(D)<\infty$; then $D$ is bounded, and all estimates below may be applied with $M_R$ replaced by $M(D)$.

		If $0<z\le64\beta^{-2}$, the result follows from $v\le1$ after increasing the constant. Assume $z>64\beta^{-2}$ and set
		\[
			\ell:=\frac\beta{16}\sqrt t.
		\]
		The intrinsic cylinder centered at $(t,x)$ with spatial radius $\ell$ is contained in
		$(0,(1+\beta)t)\times A_\beta$. Lemma~\ref{lem:mean1}, followed by enlarging the cylinder, gives
		\begin{align*}
			v(t,x)^2
			& \le\frac{Ce^{CM(D)}}{\ell^2\mu_W(D_W(x,\ell))}
			\int_{t/2}^{(1+\beta)t}\!\int_{A_\beta}v(s,y)^2\,\dd\mu_W(y)\,\dd s \\
			&\le \frac{Ce^{CM(D)}t}{\ell^2\mu_W(D_W(x,\ell))}
			\int_{A_\beta}v((1+\beta)t,y)^2\,\dd\mu_W(y),
	\end{align*}
where we use the monotonicity of $s\mapsto v(s,y)$ in the last inequality.
		Then, by the volume-ratio estimate \eqref{vr} that,
		\[
			\frac1{\ell^2\mu_W(D_W(x,\ell))}
			\le\frac{Ce^{CM(D)}\beta r}{\ell^3\mu_W(A_\beta)},
		\] 
		we obtain
		\begin{align}\label{eq:combine1}
			v(t,x)^2
			\le\frac{C_\beta e^{CM(D)}\sqrt z}{\mu_W(A_\beta)}
			\int_{A_\beta}v((1+\beta)t,y)^2\,\dd\mu_W(y).
		\end{align}
		On the other hand, note that the intrinsic distance from $A_\beta$ to $D^c$ is $R_\beta$. Lemma~\ref{lem:L2exit}, applied at time $(1+\beta)t$, and the volume-ratio estimate yield
		\begin{align*}
			\frac1{\mu_W(A_\beta)}
			\int_{A_\beta}v((1+\beta)t,y)^2\,\dd\mu_W(y)
			&\le C_\beta e^{CM(D)}
			\left(1+z+z^{-1}\right)\\
			&\quad\times
			\exp\left\{-\frac{(1-\beta)^2}{1+\beta}z\right\}.
		\end{align*}
		Combining this with \eqref{eq:combine1} and taking square roots gives
		\[
			v(t,x)
			\le C_\beta e^{CM(D)}
			z^{1/4}(1+z+z^{-1})^{1/2}
			\exp\{-a_\beta z\},
			\qquad
			a_\beta:=\frac{(1-\beta)^2}{2(1+\beta)}.
		\]
		
		Because $a_\beta>a_\theta$ and $z>1$, the polynomial factor is absorbed by the exponential gap:
		\[
			z^{1/4}(1+z+z^{-1})^{1/2}e^{-a_\beta z}
			\le C_{\beta,\theta}e^{-a_\theta z}.
		\]
		Thus
		\[
			P_x^W(\tau_{D_W(x,r)}\le t)
			\le C_\theta e^{C_\theta M(D_W(x,r))}
			\exp\left\{-a_\theta\frac{r^2}{t}\right\}.
		\]
		This implies \eqref{eq:local-exit}.
	\end{proof}
	
	As a corollary of Theorem~\ref{thm:local-exit}, we immediately have the following.
	
	\begin{corollary}[Euclidean exit interval]\label{cor:euclidean}
		Fix $\theta\in(0,1)$. Set
		\begin{align*}
		r_W(x,R):=\min\Biggl\{&
		\int_{x-R}^{x}\sqrt{\frac{\exp(\rho(z,W)-a(z,W))}{2}}\,\dd z,\\
		&\int_x^{x+R}\sqrt{\frac{\exp(\rho(z,W)-a(z,W))}{2}}\,\dd z
		\Biggr\}.
		\end{align*}
		There exists a positive constant $C_{\theta}$ such that, for every $x\in\R$, $R>0$, and $0<t\le1$,
		\begin{equation*}
			P_x^W\left(\tau_{(x-R,x+R)}\le t\right)
			\le C_{\theta}e^{C_{\theta}M((x-R,x+R))}
			\exp\left\{-a_\theta\frac{r_W(x,R)^2}{t}\right\}.
		\end{equation*}
	\end{corollary}
	
	\begin{proof}
		Put $r=r_W(x,R)$. By the definition of $r_W(x,R)$,
		$D_W(x,r)\subset(x-R,x+R)$, and hence
		\[
			\{\tau_{(x-R,x+R)}\le t\}
			\subset\{\tau_{D_W(x,r)}\le t\}.
		\]
		The conclusion follows from Theorem~\ref{thm:local-exit} and
		$M(D_W(x,r))\le M((x-R,x+R))$.
	\end{proof}

%

	Based on Theorem~\ref{thm:local-exit}, we obtain Proposition~\ref{prop:product}.
	\begin{proof}[Proof of Proposition~\ref{prop:product}]
		Fix $\theta\in(0,1)$. For all sufficiently small $\varepsilon$, Theorem~\ref{thm:local-exit} and the local property of $X^W$ give that, for any $|x|\le R$,
		\[
		P_x^W(\tau_{D_W(x,r_i)}\le\varepsilon h_i; \sup_{t \le \eps h_i} |X^W(t)| \le R)
		\le C_{\theta}e^{C_{\theta}M_{R}}
		\exp\left\{-a_\theta\frac{r_i^2}{\varepsilon h_i}\right\}.
		\]
		Multiplying over $i$ and taking expectation in the environment yields
		\begin{align*}
			&\quad\E\left[\prod_{i=1}^n\sup_{|x|\le R}P_x^W(\tau_{D_W(x,r_i)}\le\varepsilon h_i; \sup_{t \le \eps h_i} |X^W(t)| \le R)\right]\\
			&\le C_{\theta}^n \left(\prod_{i=1}^nh_i^{-N}\right)
			\E e^{nC_{\theta}M_{R}}
			\exp\left\{-\frac{a_\theta}{\varepsilon}\sum_{i=1}^n\frac{r_i^2}{h_i}\right\}.
		\end{align*}
		By assumption $H_2$, the expectation is finite. 
		Since the $h_i$ are fixed,
		\[
			\varepsilon\log\left(
			C_{\theta}^n \prod_{i=1}^nh_i^{-N}
			\E e^{nC_{\theta}M_R}\right)\longrightarrow0.
		\]
		Multiplying by $\varepsilon$, letting $\varepsilon\downarrow0$, and finally
		letting $\theta\downarrow0$ proves \eqref{eq:product-bound}.
	\end{proof}

	\section{Annealed upper bound}\label{sec:annealed-upper}
	In this section, we prove the annealed upper bound via the finite-dimensional estimates and exponential tightness in the same spirit as the method in \cite{DemboZeitouni2010}, taking extra care of the random environments.

	\subsection{Exponential tightness}\label{sec:exponential-tightness}
	Based on the local exit estimate of Theorem~\ref{thm:local-exit}, we prove exponential tightness of the annealed processes. First, we show that the annealed processes are exponentially bounded.
	
	\begin{lemma}\label{lem:compact-containment}
		Assume $H_1$--$H_3$.  One has
		\begin{equation}\label{eq:compact-containment}
			\lim_{R\to\infty}\limsup_{\varepsilon\downarrow0}\varepsilon
			\log P_0^{\ann}\left(\sup_{0\le t\le1}|X_{\varepsilon}(t)|>R\right)=-\infty.
		\end{equation}
	\end{lemma}
	
	\begin{proof}
		For fixed $R$, Corollary~\ref{cor:euclidean} gives
		\begin{align*}
			P_0^{\ann}\left(\sup_{0\le t\le1}|X_\varepsilon(t)|>R\right)
			&\le C_{\theta} \E\left[e^{C_{\theta}M_R}
			\exp\left\{-a_\theta\frac{r_W(0,R)^2}{\varepsilon}\right\}\right]\\
			&\le  C_{\theta} \E[e^{C_{\theta} M_R}]
			\exp\left\{-\frac{a_\theta}{\varepsilon}
			\operatorname{ess\,inf}_{W}r_W(0,R)^2\right\}.
		\end{align*}
		Assumption $H_2$ makes the expectation finite. Hence
		\[
			\limsup_{\varepsilon\downarrow0}\varepsilon\log
			P_0^{\ann}\left(\sup_{0\le t\le1}|X_\varepsilon(t)|>R\right)
			\le-a_\theta\operatorname{ess\,inf}_{W}r_W(0,R)^2.
		\]
		Letting $R\to\infty$ and using $H_3$ proves \eqref{eq:compact-containment}.
	\end{proof}
	\begin{proposition}[Exponential tightness]\label{prop:exp-tight}
	Assume $H_1$--$H_4$. For every $L\ge1$, there exists a deterministic
	compact set $K_L\subset C_0([0,1])$ such that
	\[
		\limsup_{\varepsilon\downarrow0}\varepsilon\log
		P_0^{\ann}(X_\varepsilon\notin K_L)\le-L.
	\]
\end{proposition}

\begin{proof}
	Fix $L>0$ and $\alpha\in(0,1/2)$. By
	Lemma~\ref{lem:compact-containment}, choose $R$ so large that
	\[
		\limsup_{\varepsilon\downarrow0}\varepsilon\log
		P_0^{\ann}(\|X_\varepsilon\|_\infty>R)\le-L.
	\]
	For $m\ge1$, set $\Delta_m=2^{-m}$ and $r_m=2^{-\alpha m}$. The Markov property and Theorem \ref{thm:local-exit} give
	\begin{align*}
		&\quad 	  P_0^{\ann}\left(
	\max_{0\le k<2^m}\sup_{0\le u\le\Delta_m}
	d_W(X_\varepsilon(k\Delta_m+u),X_\varepsilon(k\Delta_m))>r_m,
	\ \|X_\varepsilon\|_\infty\le R
	\right)
	\\
	&\le
	\sum_{k=0}^{2^m-1}
	\E\Bigg[
	P_0^W\left(
	\sup_{0\le u\le\Delta_m}
	d_W(X_\varepsilon(k\Delta_m+u),X_\varepsilon(k\Delta_m))>r_m,
	\ \|X_\varepsilon\|_\infty\le R
	\right)
	\Bigg]
	\\
	&\le
	\sum_{k=0}^{2^m-1}
	\E\left[
	\sup_{|x|\le R}
	P_x^W\left(
	\tau_{D_W(x,r_m)}\le \varepsilon\Delta_m,\,
	\sup_{s\le\varepsilon\Delta_m}|X_s^W|\le R
	\right)
	\right]
	\\
	&\le
	2^m\E\left[
	\sup_{|x|\le R}
	P_x^W\left(
	\tau_{D_W(x,r_m)}\le \varepsilon\Delta_m,\,
	\sup_{s\le\varepsilon\Delta_m}|X_s^W|\le R
	\right)
	\right]\\
		& \le C\exp\left\{-\frac{c2^{m(1-2\alpha)}}{\varepsilon}\right\},
	\end{align*}
%
for all sufficiently small $\varepsilon$, uniformly in $m\ge1$.
	Choose $m_0$ so large that $c2^{m_0(1-2\alpha)}/2>L$.

	Let $K_L$ be the set of all $f\in C_0([0,1])$ such that
	$\|f\|_\infty\le R$ and for some $\Lambda\in\mathscr L$,
			\begin{equation}\label{hc}
				\max_{0\le k<2^m}
		\sup_{k2^{-m}\le t\le(k+1)2^{-m}}
		|\Lambda(f(t))-\Lambda(f(k2^{-m}))|
		\le2^{-\alpha m}, \ \forall m\ge m_0.
		\end{equation}

The set $K_L$ is compact. Indeed, let $(f_j)\subset K_L$ and choose
	witnesses $\Lambda_j\in\mathscr L$. Put
	$h_j:=\Lambda_j\circ f_j$. By the assumption $H_4$, we see that there is a strictly
	increasing function $\Lambda \in C(\R)$ such that, after passing to a subsequence, $\Lambda_j\to\Lambda$ locally uniformly. Furthermore,  the uniform bound $\|f_j\|_{\infty} \le R$ together with \eqref{hc} allows us to take a further subsequence so that
	$h_j\to h$ uniformly for some $h \in C_0([0,1])$. Therefore, Lemma~\ref{lem:inverse-stability} gives
	\[
		f_j=\Lambda_j^{-1}\circ h_j
		\longrightarrow\Lambda^{-1}\circ h
		\quad\text{uniformly}.
	\]
	The defining inequalities pass to the limit, so $K_L$ is compact.

Since $d_W(x,y)=|\Lambda_W(y)-\Lambda_W(x)|$ and $\Lambda_W\in\mathscr L$ for
	$\Pbb$-a.e. $W$, summing over $m\ge m_0$, we obtain
	\begin{align*}
		P_0^{\ann}(X_\varepsilon\notin K_L)
		&\le
		P_0^{\ann}(\|X_\varepsilon\|_\infty>R)
		+
		C\sum_{m=m_0}^{\infty}
		\exp\left\{-\frac{c2^{m(1-2\alpha)}}{\varepsilon}\right\}
		\\
		&\le
		P_0^{\ann}(\|X_\varepsilon\|_\infty>R)
		+
		C\exp\left\{-\frac{c2^{m_0(1-2\alpha)}}{2\varepsilon}\right\}.
	\end{align*}
	Therefore,
	\[
	\limsup_{\varepsilon\downarrow0}
	\varepsilon\log
	P_0^{\ann}(X_\varepsilon\notin K_L)
	\le -L,
	\]
	which proves the assertion.
\end{proof}

	\subsection{Finite-dimensional upper estimates}\label{Finite-dimensional}

For a partition $\pi=\{0=t_0<t_1<\cdots<t_n=1\}$,
$x=(x_1,\ldots,x_n)\in\R^n$, and $x_0=0$, define
\[
	J_{\Lambda,\pi}(x)
	:=\frac12\sum_{i=1}^n
	\frac{|\Lambda(x_i)-\Lambda(x_{i-1})|^2}{t_i-t_{i-1}},
	\qquad \Lambda\in\mathscr L.
\]
For the actual environment, write $J_{W,\pi}:=J_{\Lambda_W,\pi}$.

\begin{proposition}[Finite-dimensional upper bound]\label{prop:fd-upper}
	Assume $H_1$--$H_4$. Let $\pi$ be as above and let
	$F\subset\R^n$ be closed. Then
	\[
		\limsup_{\varepsilon\downarrow0}\varepsilon\log
		P_0^{\ann}\bigl(
		(X_\varepsilon(t_1),\ldots,X_\varepsilon(t_n))\in F
		\bigr)
		\le-\inf_{\substack{x\in F\\ \Lambda\in\mathscr L}}
		J_{\Lambda,\pi}(x).
	\]
\end{proposition}

\begin{proof}
	We first assume that $F$ is compact. Choose $R_0$ such that
	$F\subset[-R_0,R_0]^n$ and fix $R>R_0+1$. Compactness of
	$\mathscr L$ in $C(\R)$ gives equicontinuity on
	$[-R_0-1,R_0+1]$. Hence, for every $\eta>0$, there exists
	$r\in(0,1/2)$ such that
	\[
		\sup_{\Lambda\in\mathscr L}
		\sup_{\substack{x,y\in[-R_0-1,R_0+1]\\ |x-y|\le2r}}
		|\Lambda(x)-\Lambda(y)|\le\frac\eta4.
	\]
	Choose finitely many centers $a^1,\ldots,a^m\in F$ such that the cubes
	\[
		Q^j:=\prod_{i=1}^nQ_i^j:=\prod_{i=1}^n[a_i^j-r,a_i^j+r]
	\]
	cover $F$. Put $a_0^j:=0$ and $Q_0^j:=\{0\}$. For
	$\Lambda\in\mathscr L$, define
	\[
		\delta_i^j(\Lambda):=
		\inf\{d_{\Lambda}(u,v):u\in Q_{i-1}^j,\ v\in Q_i^j\},
	\]
	and set
	\[
		A_j(\Lambda):=\sum_{i=1}^n
		\frac{\delta_i^j(\Lambda)^2}{t_i-t_{i-1}}.
	\]
	The map $A_j:\mathscr L\to[0,\infty)$ is continuous.

	Set $h_i:=t_i-t_{i-1}$ and, for the fixed environment $W$, write
	$\delta_i^j:=\delta_i^j(\Lambda_W)$. On the event
	\[
		E_j^R:=\left\{X_\varepsilon(t_i)\in Q_i^j,\ 1\le i\le n;
		\ \|X_\varepsilon\|_\infty\le R\right\},
	\]
	the following crossing inclusion holds for every $i$ with
	$\delta_i^j>0$:
	\[
		E_j^R\subset
		\left\{
		\inf\left\{s\ge0:
		X^W_{\varepsilon t_{i-1}+s}\notin
		D_W\bigl(X^W_{\varepsilon t_{i-1}},\delta_i^j\bigr)
		\right\}
		\le\varepsilon h_i
		\right\}.
	\]

	Applying the Markov property successively at the deterministic times
	$\varepsilon t_{i-1}$, and retaining the condition that each segment
	stays in $[-R,R]$, gives
	\begin{align*}
		P_0^W(E_j^R)
		&\le
		\prod_{\{i:\delta_i^j>0\}}
		\sup_{|x|\le R}P_x^W\left(
		\tau_{D_W(x,\delta_i^j)}\le\varepsilon h_i;
		\ \sup_{0\le s\le\varepsilon h_i}|X_s^W|\le R
		\right).
	\end{align*}
	Using Proposition~\ref{prop:product} together with the
	 continuity of $A_j$, we obtain
\begin{equation}\label{Ej}
		\limsup_{\varepsilon\downarrow0}\varepsilon\log
		P_0^{\ann}\bigl(
		E_j^R \bigr)
		\le- \frac12 \min_{\Lambda\in\mathscr L}A_j(\Lambda).
\end{equation}

	For $u\in Q_{i-1}^j$ and $v\in Q_i^j$, equicontinuity gives
	\[
		d_{\Lambda}(a_{i-1}^j,u)\le\frac\eta4,
		\qquad
		d_{\Lambda}(v,a_i^j)\le\frac\eta4,
	\]
	uniformly in $\Lambda\in\mathscr L$. Hence
	\[
		\delta_i^j(\Lambda)
		\ge\bigl(d_{\Lambda}(a_{i-1}^j,a_i^j)-\eta/2\bigr)_+.
	\]
	For every $\gamma\in(0,1)$ and $\ell,c\ge0$,
	\[
		(\ell-c)_+^2\ge(1-\gamma)\ell^2
		-\frac{1-\gamma}{\gamma}c^2.
	\]
	It follows that
	\[
		A_j(\Lambda)
		\ge2(1-\gamma)J_{\Lambda,\pi}(a^j)
		-C_{\pi,\gamma}\eta^2.
	\]
	Taking the finite union over $j$ in \eqref{Ej}, then letting
	$\eta\downarrow0$, $\gamma\downarrow0$, and
	$\theta\downarrow0$, gives
	\[
		\limsup_{\varepsilon\downarrow0}\varepsilon\log
		P_0^{\ann}\bigl(
		(X_\varepsilon(t_1),\ldots,X_\varepsilon(t_n))\in F;
		\ \|X_\varepsilon\|_\infty\le R\bigr)
		\le-\inf_{\substack{x\in F\\ \Lambda\in\mathscr L}}
		J_{\Lambda,\pi}(x).
	\]
	Letting $R\to\infty$ and using
	Lemma~\ref{lem:compact-containment} proves the compact case.

	For a general closed set, put $F_K:=F\cap[-K,K]^n$. Apply the compact
	case to $F_K$, use compact containment for its complement, and let
	$K\to\infty$.
\end{proof}

	\subsection{Path-space upper bound for closed sets}\label{Path-space}

	To prove the path-space upper bound for closed sets, we need the following general Schilder-type lemma.
	
	\begin{lemma}
		\label{lem:partition-energy}
		Given a compact set $\mathcal{H} \subset C_0([0,1])$, we have
		\begin{equation}\label{eq:minimax-partition}
			\inf_{h \in \mathcal{H}} \int_0^1h'(s)^2 \dd s =\sup_\pi  \inf_{h \in \mathcal{H}} \sum_{i=1}^n\frac{\left (\int_{t_{i-1}}^{t_i} h'(s) \dd s \right)^2}{t_i-t_{i-1}},
		\end{equation}
		where the supremum runs over all finite partitions
		$\pi=\{0=t_0<t_1<\cdots<t_n=1\}$.
	\end{lemma}
	
	\begin{proof}
		The inequality ``$\le$'' follows immediately from Schilder's theorem (see \cite[Theorem~5.2.3]{DemboZeitouni2010}) that
		\begin{equation}\label{Schlider}
			\int_0^1h'(s)^2 \dd s=\sup_\pi \sum_{i=1}^n\frac{\left (\int_{t_{i-1}}^{t_i} h'(s) \dd s \right)^2}{t_i-t_{i-1}}=: \sup_\pi J_{\pi}(h), \ \forall h \in C_0([0,1]),
		\end{equation}
		where the supremum runs over all finite partitions.
		
		For the reverse inequality, let $a<\inf_{h \in \mathcal{H}} \int_0^1h'(s)^2 \dd s$. For every $h \in \mathcal{H}$, \eqref{Schlider} provides a finite partition $\pi_h$ such that
		\[
		J_{\pi_h}(h)>a.
		\]
		Note that for a fixed finite partition $\pi$, the map $h \mapsto J_{\pi}(h)$ is continuous on $C_0([0,1])$. Compactness of $\mathcal{H}$ gives finite neighborhoods $U_i$, corresponding to finite partitions $\pi_i$, $1 \le i \le m$, such that $\cup_{1 \le i \le m}U_i$ covers $\mathcal{H}$ and
		\[
		J_{\pi_i}(h)>a, \ \forall h \in U_i.
		\]
		Let $\pi_*$ be a common refinement. Because Cauchy's inequality gives
		\[
		\frac{(x_k-x_i)^2}{t_k-t_i}
		\le \frac{(x_j-x_i)^2}{t_j-t_i}+
		\frac{(x_k-x_j)^2}{t_k-t_j},
		\quad (t_i<t_j<t_k),
		\]
		we have $J_{\pi_*}\ge J_{\pi_j}$ on every path, and therefore
		\[
		\inf_{h \in \mathcal{H}}J_{\pi_*}(h)\ge a.
		\]
		Taking the supremum over $\pi_*$ and then letting $a\uparrow \inf_{h \in \mathcal{H}} \int_0^1h'(s)^2 \dd s$ proves \eqref{eq:minimax-partition}.
	\end{proof}
	
	For $f\in C_0([0,1])$, use the shorthand
	\[
		J_{\Lambda,\pi}(f):=
		J_{\Lambda,\pi}\bigl(f(t_1),\ldots,f(t_n)\bigr)
		=\frac12J_\pi(\Lambda\circ f).
	\]
	\begin{lemma}[Environment--partition identification]\label{lem:environment-partition}
	Assume $H_4$. If $F\subset C_0([0,1])$ is compact, then
	\[
		\sup_\pi\inf_{\substack{f\in F\\ \Lambda\in\mathscr L}}
		J_{\Lambda,\pi}(f)
		=\inf_{f\in F}I(f).
	\]
\end{lemma}

\begin{proof}
	The set
	\[
		\mathcal H(F):=
		\{\Lambda\circ f:\Lambda\in\mathscr L,\ f\in F\}
	\]
	is compact in $C_0([0,1])$. Indeed, $F$ has bounded range and the
	composition map is continuous for locally uniform convergence of
	$\Lambda$ and uniform convergence of $f$. Moreover,
	\[
		J_{\Lambda,\pi}(f)=\frac12J_\pi(\Lambda\circ f).
	\]
	Applying Lemma~\ref{lem:partition-energy} to $\mathcal H(F)$ gives
	\begin{align*}
		\sup_\pi\inf_{\substack{f\in F\\ \Lambda\in\mathscr L}}
		J_{\Lambda,\pi}(f)
		&=\inf_{h\in\mathcal H(F)}\mathcal I_{\mathrm{CM}}(h)\\
		&=\inf_{f\in F}\min_{\Lambda\in\mathscr L}
		\mathcal I_{\mathrm{CM}}(\Lambda\circ f)\\
		&=\inf_{f\in F}I(f).
	\end{align*}
\end{proof}

	Now we can derive the upper bound \eqref{eq:ldp-upper}.

	\begin{proposition}\label{P5.5}
		Assume $H_1$--$H_4$ hold. Then for every closed set
		$F\subset C_0([0,1])$,
		\[
			\limsup_{\varepsilon\downarrow0}
			\varepsilon\log P_0^{\ann}(X_\varepsilon\in F)
			\le-\inf_{\varphi\in F}I(\varphi).
		\]
	\end{proposition}

	\begin{proof}
		Fix $L>0$ and let $K_L$ be the deterministic compact set from
		Proposition~\ref{prop:exp-tight}. Put $F_L:=F\cap K_L$. Then
		\[
			P_0^{\ann}(X_\varepsilon\in F)
			\le P_0^{\ann}(X_\varepsilon\in F_L)
			+P_0^{\ann}(X_\varepsilon\notin K_L).
		\]
		For every finite partition $\pi$, Proposition~\ref{prop:fd-upper} applied to
		the compact finite-dimensional image of $F_L$ gives
		\[
			\limsup_{\varepsilon\downarrow0}\varepsilon\log
			P_0^{\ann}(X_\varepsilon\in F_L)
			\le-
			\inf_{\substack{f\in F_L\\ \Lambda\in\mathscr L}}
			J_{\Lambda,\pi}(f).
		\]
		Taking the supremum over $\pi$ and applying
		Lemma~\ref{lem:environment-partition},
		\[
			\limsup_{\varepsilon\downarrow0}\varepsilon\log
			P_0^{\ann}(X_\varepsilon\in F_L)
			\le-\inf_{f\in F_L}I(f)
			\le-\inf_{f\in F}I(f).
		\]
		Therefore
		\[
			\limsup_{\varepsilon\downarrow0}\varepsilon\log
			P_0^{\ann}(X_\varepsilon\in F)
			\le\max\left\{-\inf_{f\in F}I(f),-L\right\}.
		\]
		If $\inf_FI<\infty$, choose $L>\inf_FI$; otherwise let
		$L\to\infty$.
	\end{proof}

\section{Appendix}

The aim of this section is to verify that for each fixed $W$, the Itô--McKean construction of $X^W$ from \eqref{eq:brox-time-change-section2} is associated with the generator ${\mathcal L}_W$.

%
	
	\begin{lemma}\label{l3.1}
		Assume $H_1$. Let $X^W$ be the diffusion given by the scale-time construction \eqref{eq:brox-time-change-section2}, and let $(P_t^W)_{t\geq0}$ be its transition semigroup. Then $(P_t^W)_{t\geq0}$ is the $L^2(\mu_W)$-semigroup associated with $(\mathcal E_W,\mathcal F_W)$.
	\end{lemma}
	
	\begin{proof}
		Fix the environment $W$ and a starting point $x$. Recall that 
		\[
			\phi_W(t):=\int_0^t\sigma_W(B(s)+S_W(x))^{-2}\,\dd s.
		\]
		Define the diffusion process starting from $S_W(x)$ by
			\[
			Y_W(t):=B({\phi_W^{-1}(t)})+S_W(x).
		\]
		According to the time-change theorem for regular Dirichlet forms (see \cite[Theorem 6.2.1]{FukushimaOshimaTakeda2011}), $Y_W$ is associated on $L^2(\R,\nu)$ with
		\begin{equation*}
			\widehat{\mathcal E}(u,v)
			:=\frac12\int_{\R}u'(q)v'(q)\,\dd q,
			\qquad
			\nu(\dd q):=\sigma_W(q)^{-2}\,\dd q,
		\end{equation*}
		and
		\begin{equation*}
			\widehat{\mathcal F}
			=\left\{u\in L^2(\R,\nu):
			u\text{ is locally absolutely continuous and }u'\in L^2(\R,\dd q)\right\}.
		\end{equation*}

		For $f$ on $\R$, put $\widehat f(q):=f(S_W^{-1}(q))$. With $q=S_W(z)$,
		\[
			\widehat f'(S_W(z))=e^{a(z,W)}f'(z),
			\qquad
			\dd q=e^{-a(z,W)}\,\dd z.
		\]
		Consequently,
		\begin{align*}
			\widehat{\mathcal E}(\widehat f,\widehat g)
			&=\frac12\int_{\R}e^{a(z,W)}f'(z)g'(z)\,\dd z
			=\frac12\mathcal E_W(f,g).
		\end{align*}
	    and
		\[
			\nu(\dd q)
			=\frac12e^{\rho(z,W)}\,\dd z
			=\frac12\mu_W(\dd z).
		\]
		The pullback of $(\widehat{\mathcal E},\widehat{\mathcal F})$ is therefore
		$(\frac12\mathcal E_W,\mathcal F_W)$ on $L^2(\frac12\mu_W)$. Multiplying both the form and the reference measure by the same positive constant leaves the associated generator and semigroup unchanged. Hence
		\[
			X^W(t)=S_W^{-1}(Y_W(t))
		\]
		is the diffusion associated with $(\mathcal E_W,\mathcal F_W)$ on $L^2(\mu_W)$.
	\end{proof}

As a consequence, we see that exit probabilities are weak solutions.
	
	\begin{lemma}\label{lem:exit-caloric}
		Assume $H_1$. Let $D=(\ell,r) $ be a bounded domain and set
		\[
		v(s,x):=P_x^W(\tau_D\le s),\quad (s,x) \in (0,S)\times D.
		\]
		Then $v$ is a continuous weak solution on the cylinder $(0,S)\times D$.
	\end{lemma}
	
	\begin{proof}
	Since
		\begin{equation*}
			P_t^Df(x):=E_x^W\!\left[f(X_t^W);\,t<\tau_D\right]
		\end{equation*}
		is the semigroup of the process killed upon exiting $D$, by Lemma \ref{l3.1} and Theorem 4.4.2 in \cite{FukushimaOshimaTakeda2011}, the semigroup $P_t^D$ is associated with the part form
		\begin{equation*}
			\mathcal E_W^D(f,g)=\mathcal E_W(f,g),
			\qquad  \mathcal F_W^D
			:=\overline{C_c^\infty(D)}^{\,\mathcal E_{W,1}}
			=\{f\in\mathcal F_W:\widetilde f=0
			\text{ quasi-everywhere on }D^c\}.
		\end{equation*}
		Because $D$ is bounded and $a,\rho$ are continuous, the $\mathcal E_{W,1}$-norm and the usual $H^1(D)$-norm are equivalent on $C_c^\infty(D)$. It follows that $\mathcal F_W^D=H_0^1(D)$. Furthermore, 
		Since $\mathbf 1_D\in L^2(D,\mu_W)$, the spectral theorem implies that, for every $t>0$,
		\[
		h(t):=P_t^D\mathbf 1_D\in H_0^1(D),
		\]
		and for every $\psi\in H_0^1(D)$,
		\begin{equation*}
			\frac{d}{dt}\langle h(t),\psi\rangle_{L^2(\mu_W)}
			=-\mathcal E_W^D(h(t),\psi).
		\end{equation*}
		Therefore, $
			v(t,x)=1-h(t,x) $
		is a weak solution to equation (\ref{02}) on $(0,S)\times D$

		Next, we prove the continuity of \(v\) on \((0,S)\times D\).
		By the argument in Lemma \ref{l3.1}, we see that $Y_W$ is associated with the quadratic form $\widehat{\mathcal E}$ on $L^2(\R,\nu)$.
	Set $\tau:=\inf\{t \ge 0: Y_W(t) \notin S_W(D)\}$. Then, by Theorem 4.4.2 in \cite{FukushimaOshimaTakeda2011}, the killed semigroup of \(Y_W\), defined by
		\begin{equation*}
			T_t^{S_W(D)}f(x):=E_x^W\!\left[f(Y_W(t));\,t<\tau\right]
		\end{equation*}
		 is associated with the part form $\widehat{\mathcal E}$ on $S_W(D)$. 
		Consequently, its generator is given by
		$$\mathcal A f(x)=\frac{1}{2}\sigma_W(x)^2f''(x).$$
		 Combining this with the Harnack inequality, it follows that for every bounded function $f$, the map $(t,x) \mapsto T_t^{S_W(D)}f(x)$ is continuous on $(0,S)\times S_W(D)$. Taking \(f=\mathbf 1_{S_W(D)}\) and using the relation \(X^W(t)=S_W^{-1}(Y_W(t))\), we obtain
\[
v(s,x)=1-\bigl(T_s^{S_W(D)}\mathbf 1_{S_W(D)}\bigr)\bigl(S_W^{-1}(x)\bigr)
\]
 is continuous on $(0,S)\times D$.
	\end{proof}
	
\vskip 0.3cm

	\medskip
	
	\noindent{\bf Acknowledgements}\ This work is partly supported by National Key R\&D Program of China (No. 2022YFA1006000), National Natural Science Foundation of China (Nos. 12131019, 12201598, 12371151) and the Fundamental Research Funds for the Central Universities, China (Nos. WK0010250108, WK0010000081).

\end{document}